\documentclass[twocolumn]{elsarticle}
\pdfoutput=1

\usepackage{graphicx}
\usepackage{amsmath}
\usepackage{amsthm}
\usepackage{amsfonts}
\usepackage{amssymb}
\usepackage{float}

\theoremstyle{plain}
\newtheorem{theorem}{Theorem}[section]
\newtheorem*{theorem*}{Theorem}
\newtheorem*{conjecture*}{Conjecture}
\newtheorem*{corollary*}{Corollary}

\newtheorem{corollary}[theorem]{Corollary}
\newtheorem*{claim}{Claim}

\theoremstyle{definition}
\newtheorem*{problem*}{Problem}
\newtheorem{problem}{Problem}

\theoremstyle{remark}

\newcommand{\N}{\mathbb{N}}

\begin{document}

\newcommand{\pf}{{\bf Proof: }}

\begin{raggedleft}

\advance\rightskip by-1cm
\null\vskip 1cm
{\LARGE\bf\center Arc-Disjoint Paths and Trees in 2-Regular Digraphs\par}
\vskip 0.5cm
{\LARGE\center J\o{}rgen Bang-Jensen$^\text{a}$ and Sven Simonsen$^\text{a}$\par}
{\center $^\text{a}$ Department of Mathematics and Computer Science, University of Southern Denmark, Odense DK-5230, Denmark (email:jbj@imada.sdu.dk,svsim@imada.sdu.dk).\par}
\vskip \lineskip
{\Large\center \today \par}

\end{raggedleft}

\vskip 1.5cm
{\large\bf\center Abstract \par}

An out-(in-)branching $B_s^+ (B_s^-)$ rooted at $s$ in a digraph $D$ is a connected spanning subdigraph of $D$ in which every vertex $x\not =s$ has precisely one arc entering (leaving) it and $s$ has no arcs entering (leaving) it. We settle the complexity of the following two problems: 

\begin{itemize}

\item Given a 2-regular digraph $D$, decide if it contains two arc-disjoint branchings $B^+_u$, $B^-_v$. 

\item Given a 2-regular digraph $D$, decide if it contains an out-branching $B^+_u$ such that $D$ remains connected after removing the arcs of $B^+_u$.

\end{itemize}

Both problems are NP-complete for general digraphs \cite{bangJCT51,bangTCSta}. We prove that the first problem remains  NP-complete for  2-regular digraphs, whereas the second problem turns out to be polynomial when we do not prescribe the root in advance. We also prove that, for 2-regular digraphs, the latter problem is in fact equivalent to deciding if $D$ contains two arc-disjoint out-branchings. We generalize this result to k-regular digraphs where we want to find a number of pairwise arc-disjoint spanning trees and out-branchings such that there are k in total, again without prescribing any roots.


\section{Introduction}

Every digraph will be finite. Notation will follow \cite{Bang} unless stated otherwise. We recall the most relevant concepts below:

For a given digraph $D=(V,A)$ the  {\bf in-degree} $d^-(X)$ ({\bf out-degree} $d^+(X)$) of the vertex set $X\subset V$ is the number of arcs entering (leaving) $X$. A digraph is {\bf $k$-arc-strong} if $d^+(X)\geq k$ for every non-empty proper subset $X$ of $V$.

When we \textbf{split} a vertex $v\in V$ into it's ingoing part $v^-$ and outgoing part $v^+$ we replace $v$ by  two new vertices $v^-$ and $v^+$ and replace every arc  $uv\in A$ ($vw\in A$) by the arc $uv^-$ ($v^+w$).

A digraph $D=(V,A)$ is \textbf{k-regular} if every vertex $v\in V$ has out-degree $d^+(v)=k$ and in-degree $d^-(v)=k$.

A \textbf{Hamiltonian path (cycle)} in $D$ is a directed path (cycle) that contains all vertices in $V$.

An \textbf{out-(in-)branching} $B_s^+ (B_s^-)$ rooted at the vertex $s$ in $D$ is a connected spanning subdigraph of $D$ in which every vertex $x\not =s$ has precisely one arc entering (leaving) it and $s$ has no arcs entering (leaving) it.

A \textbf{spanning tree} in $D$ is a spanning tree in the underlying graph $UG(D)$, that is the graph that appears when we disregard the direction of all arcs in $D$, turning them into edges. Note that $UG(D)$ may contain parallel edges. We say that a digraph $D$ is {\bf connected} when $UG(D)$ is a connected graph. 

Given a digraph $D=(V,A)$, a vector $r: V\rightarrow \N_0$, that maps vertices to integers, is called a \textbf{root vector} of $D$ if there exists a set of arc-disjoint out-branchings in $D=(V,A)$ such that each vertex $v\in V$ is the root of exactly $r(v)$ of these out-branchings. We extend  $r$ to subsets of $V$ by letting $r(X)=\sum_{v\in X}{r(v)}$ for all $X\subseteq V$. Note that $r(V)$ then denotes the total number of roots of out-branchings  prescribed by $r$, assuming $r$ is a root vector.

Graphs that contain k edge-disjoint spanning trees were  characterized by Tutte

\begin{theorem}[Tutte's Tree Packing Theorem] \cite{Tutte}\\ \label{tut}
A graph $G=(V,E)$ contains $k$ edge-disjoint spanning trees if and only if
\begin{equation*} e_{\cal F}\geq k(t-1) \end{equation*}
holds for every partition ${\cal F}=V_1,V_2,\ldots,V_t$ of $V$, where $e_{\cal F}$ denotes the number of edges connecting different sets $V_i,V_j$.
\end{theorem}

A maximum collection of edge-disjoint spanning trees can  be found in polynomial time by converting the problem to a matroid problem and then applying Edmonds' Matroid Partition Algorithm as explained in \cite{Edmonds2}.

Digraphs that contain k arc-disjoint out-branchings have been characterized by Edmonds' Branching Theorem \cite{Edmonds} from which one easily gets the following characterization of root-vectors:

\begin{theorem} (T 2.12 in \cite{Berczi})\\ \label{root}
Let $D=(V,A)$ be a directed graph and $r:V\rightarrow \N_0$ a vector with $r(V)=k$.

Then $r$ is a root vector if and only if
\begin{equation*} \label{eq:root} d_D^-(X)\geq k-r(X) \quad \text{for all non-empty } X\subseteq V. \end{equation*}
\end{theorem}

Edmonds \cite{Edmonds} gave a polynomial algorithm for finding these k arc-disjoint out-branchings given the roots specified by the root vector, while Frank \cite{Frank} gave an algorithm for finding the maximum number of arc-disjoint out-branchings when we do not fix the roots.

In this paper we will study two related problems, both of which are NP-complete for general digraphs, and examine them on the restricted class of 2-regular digraphs.

\begin{problem}\label{ioB}
Given a digraph $D$ and vertices $u,v$ (not necessarily distinct). Decide whether $D$ has a pair of arc-disjoint branchings $B^+_u$, $B^-_v$.
\end{problem}

Thomassen \cite{Thomassen} conjectured in 1985 that for large enough integers $r$ every $r$-arc-strong digraph contains arc-disjoint branchings $B^+_v$, $B^-_v$ for every vertex $v$. This conjecture is wide open and it is only known that $r\geq 3$ must hold. 

Thomassen also proved (see \cite{bangJCT51})  that Problem \ref{ioB} is NP-complete for general digraphs. Thus it is of interest to study Problem \ref{ioB} in special classes of digraphs. The first author proved that Problem \ref{ioB} is polynomial for Tournaments \cite{bangJCT51}. We prove below that Problem  \ref{ioB} is NP-complete for 2-regular digraphs. 

Inspired by the existence of good characterizations and algorithms (mentioned above) for respectively, the existence of $k$ edge-disjoint spanning trees in a graph and the existence of $k$ arc-disjoint out-branchings (with or without specified roots) in a digraph, Thomass\'e posed the following problem around 2006 (it appeared on the Hungarian problem page Egres open for several years). 

\begin{problem} \label{oBsT}
Given a digraph $D$, decide whether it contains an out-branching $B^+_u$ such that $D$ remains connected after removing the arcs of $B^+_u$.
\end{problem}

The first author and Yeo proved recently that Problem \ref{oBsT} is NP-complete for general digraphs \cite{bangTCSta}. In the meantime the idea of studying mixed problems where we are asking for (arc)-disjoint structures $S,T$ in a digraph $D$  where only $S$  has to respect the orientation of arcs in $D$, gave inspiration for several papers, see e.g.  \cite{bangTCS410,bangC31}.

We prove that  when we do not specify the root $u$ in Problem \ref{oBsT} and $D$ is 2-regular, then the problem  becomes equivalent to deciding if the digraph contains 2 arc-disjoint out-branchings, hence making it polynomial. This  contrasts the previous result that Problem \ref{ioB} remained NP-complete even in the class of 2-regular digraphs. The complexity of the remaining case where we do specify the root is still open. In Section \ref{related} we prove that a number of (seemingly) closely related problems are NP-complete even for 2-regular digraphs .

Packing two spanning subdigraphs in a 2-regular digraph will always require all but at most two of the arcs. So at first hand it seems that this restriction should make the problem tractable but as we mentioned above (and will prove in Sections \ref{ioBsec}, \ref{BTsec}) this is only the case for Problem \ref{oBsT}. In fact, under the assumption that P$\neq$NP, the following is a consequence of Problem \ref{ioB} being NP-complete for 2-regular digraphs (see Theorem \ref{NPinout}).

\begin{corollary} \label{useall}
It is NP-complete to decide for a given digraph $D$ on $n$ vertices and $2n-2$ arcs, whether the arcs of $D$ can be partitioned into an out branching $B^+_u$ and and in-branching $B^-_v$.
\end{corollary}

In contrast to this, it follows from our discussion in the introduction that there is a polynomial algorithm for checking whether a given set of $2n-2$ edges in a graph on $n$ vertices can be partitioned into two edge-disjoint spanning trees and there is also a polynomial algorithm for checking whether at set of $2n-2$ arcs in a digraph on $n$ vertices can be partitioned into two arc-disjoint out-branchings with or without prescribed roots.

\section{Hamiltonian Paths in 2-regular digraphs}

We begin our investigation by considering a problem that requires all arcs, namely deciding if a 2-regular digraph contains two arc-disjoint Hamiltonian cycles. We point out that all of the results in this section have been proven by Yeo earlier but the proof of Theorem \ref{hampathNP} was never published and since we use the technique illustrated below in our proof in Section \ref{ioBsec} we have included it here for completeness. We also point out that Ples\'nik \cite{plesnikIPL8} proved much earlier that both the Hamilton cycle and the Hamilton path problem are NP-complete already for planar 2-regular digraphs.

\begin{theorem}\cite[Theorem 6.1.3]{bangTCSta}\\ \label{hamcycleNP}
It is NP-complete to decide whether a given 2-arc-strong 2-regular digraph $D$ contains two arc-disjoint Hamiltonian cycles.
\end{theorem}

\begin{theorem} \label{hampathNP}
It is NP-complete to decide whether a given 2-arc-strong 2-regular digraph contains two arc-disjoint Hamiltonian paths (with any number of specified end vertices).
\end{theorem}

\begin{proof}
We will reduce from the problem of deciding if a 2-arc-strong 2-regular digraph contains two arc-disjoint Hamiltonian cycles, which according to Theorem \ref{hamcycleNP} is NP-complete.

To do this we will use the Cycle Breaker Gadget shown in Figure \ref{fig:CycleBreaker}. Notice that it is impossible to remove the arcs of a spanning $(s,t)$-path from the gadget without disconnecting some vertices.

\begin{figure}[htb]
\centering
\includegraphics[width=0.2\textwidth]{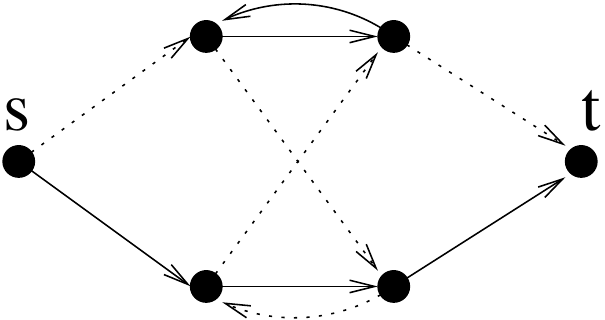}
\caption{The Cycle Breaker Gadget. One of the two possible spanning $(s,t)$-paths is highlighted. Removing it disconnects 2 vertices.}
\label{fig:CycleBreaker}
\end{figure}

Given a 2-arc-strong 2-regular digraph $D$ construct $D'$ by choosing a vertex $a\in D$ and splitting it into it's in-going part $a^-$ and it's out-going part $a^+$. Then construct $D''$ by adding a copy $G$ of the Cycle Breaker Gadget to $D'$ and identifying $a^-$ with $s$ and $t$ with $a^+$. Notice that this turns $D''$ into a 2-arc-strong 2-regular digraph.

To conclude the reduction we will argue that $D''$ contains two arc-disjoint Hamiltonian paths if and only if $D$ contains two arc-disjoint Hamiltonian cycles.

For sufficiency assume that $D$ contains arc-disjoint Hamiltonian cycles $C$ and $C'$. Then $C$ and $C'$ can be considered as Hamiltonian $(a^+,a^-)$-paths in $D'$. Now adding the two arc-disjoint path fragments $P$ and $Q$ both covering the Cycle Breaker Gadget, see Figure \ref{fig:BrokenPath}, to $C$ and $C'$ considered in $D''$ gives an arc-disjoint pair of paths, where $C+P$ is a Hamiltonian $(d,e)$-path and $C'+Q$ is a Hamiltonian $(b,c)$-path of $D'$.

\begin{figure}[htb]
\centering
\includegraphics[width=0.2\textwidth]{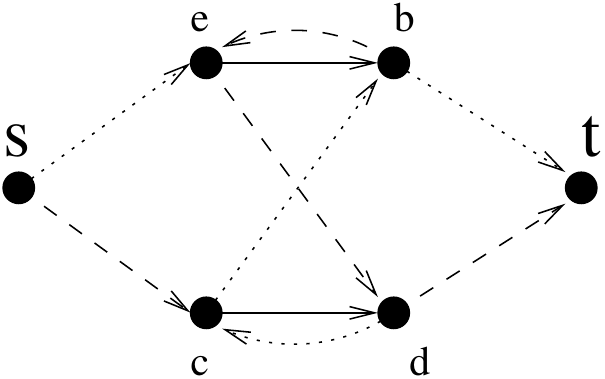}
\caption{The path fragments $P$ shown with dotted arcs and $Q$ shown with dashed.}
\label{fig:BrokenPath}
\end{figure}

Conversely let $P'$ and $Q'$ be arc-disjoint Hamiltonian paths of $D'$. Neither can contain an $(s,t)$-path since that would disconnect part of the Cycle Breaker Gadget. But that implies that both must contain a $(a^+,a^-)$-path spanning all of $D'$ translating into two arc-disjoint Hamiltonian cycles of $D$ when $a^+$ and $a^-$ are identified.

Since end vertices of the Hamiltonian paths are already forced to lie inside $G$, specifying one or more of them and requesting an arc-disjoint pair of a Hamiltonian $(d,e)$-path and a Hamiltonian $(b,c)$-path still gives the same reduction.
\end{proof}

\section{Arc-disjoint in- and out-branchings in k-regular digraphs}\label{ioBsec}

We now show that, using the Cycle Breaker Gadget, we can restrict the behavior of branchings in 2-regular digraphs. Given a 2-regular digraph $D$ we immediately see that removing two arc-disjoint branchings would leave only two arcs in $D$. Suppose we specify two vertices $u,v\in D$ and assume arc-disjoint branchings $B_u^+$ and $B_v^-$ exist.

Then the fact that both ${d^+_{B_v^-}(x)=1}$ for all ${x\in V-v}$ and ${d^+_{B_v^-}(v)=0}$ must hold implies that $v$ is the only vertex that could have two out-going arcs in $B_u^+$. Similarly only $u$ can have two in-going arcs in $B_v^-$.

So our branchings may only really "branch" on $u$ or $v$, in all other vertices they will behave just like paths.

\begin{theorem}\label{NPinout}

Problem \ref{ioB} (with or without fixed not necessarily distinct roots) is NP-complete for 2-arc-strong 2-regular digraphs.

\end{theorem}

\begin{proof}

In order to prove NP-completeness we will reduce from the problem of deciding if a 2-arc-strong 2-regular digraph $D$ contains two arc-disjoint Hamiltonian cycles.

Construct $D'$ from $D$ by choosing a vertex $a\in D$ and splitting it into it's in-going part $a^-$ and it's out-going part $a^+$. Then construct $D''$ by adding two copies $G,G'$ of the Cycle Breaker Gadget to $D'$ and identifying $a^-$ with $s$, $t$ with $s'$ and $t'$ with $a^+$, see Figure \ref{fig:DoublerootCatcher}.

\begin{figure}[htb]
\centering
\includegraphics[width=0.45\textwidth]{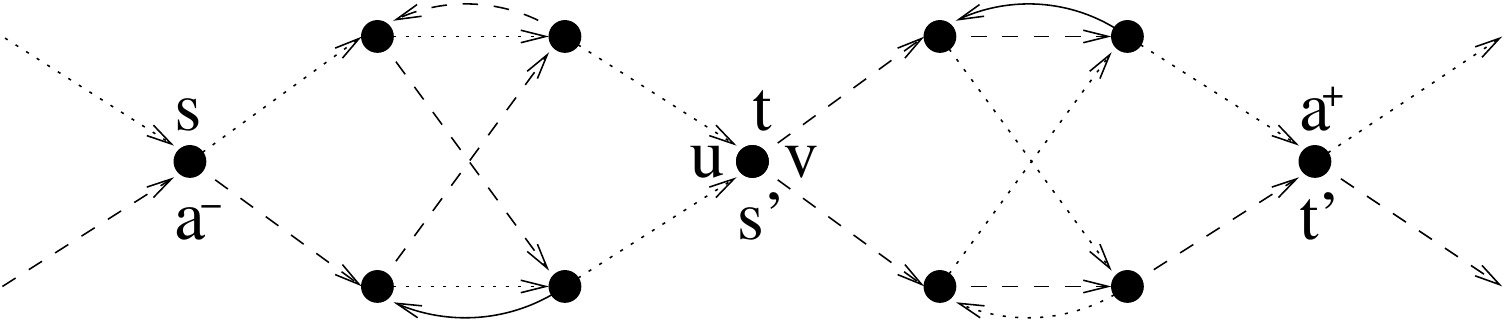}
\caption{The in-branching fragment $P$ shown with dotted arcs and the out-branching fragment $Q$ shown with dashed.}
\label{fig:DoublerootCatcher}
\end{figure}

Obviously there is no $(s,t')$-path covering all vertices of $G\cup G'$ which leaves a connected graph once the arcs of the path are removed. This implies that neither $B_u^+$ nor $B_v^-$ can consist of a covering $(s,t')$-path when restricted to $G\cup G'$. By our observation above we know that our branchings can only behave differently than a path at the root of the other branching, so both $u$ and $v$ must be in $G\cup G'$. From this follows that, when restricted to arcs of $D'$, the branchings $B_u^+$ and $B_v^-$ are arc-disjoint Hamiltonian $(a^+,a^-)$-paths in $D'$ and thus arc-disjoint Hamiltonian cycles in $D$.

Conversely, if $D$ contains two arc-disjoint Hamiltonian cycles $C$ and $C'$, then those correspond to arc-disjoint $(a^+,a^-)$-paths in $D'$ and adding branching fragments $P$ and $Q$ as in Figure \ref{fig:DoublerootCatcher} we can construct arc-disjoint in- and out-branchings in $G'$ with the same root.

\begin{figure}[htb]
\centering
\includegraphics[width=0.45\textwidth]{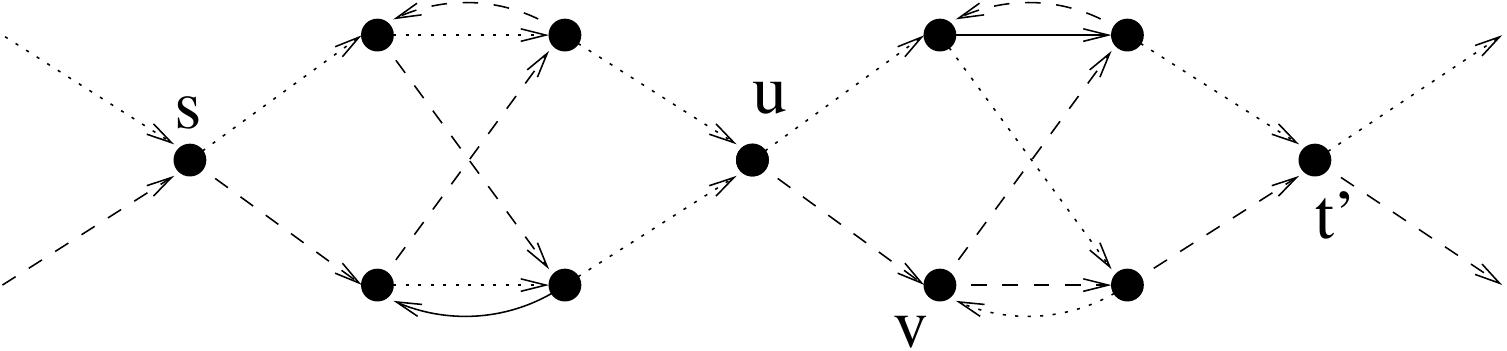}
\caption{The in-branching fragment $P$ shown with dotted arcs and the out-branching fragment $Q$ shown with dashed.}
\label{fig:DistinctRootCatcher}
\end{figure}

If we add branching fragments $P$ and $Q$ as in figure \ref{fig:DistinctRootCatcher} instead we get distinct roots.
\end{proof}

Since 2-regular digraphs are also Eulerian this directly implies:

\begin{corollary}
Problem \ref{ioB} is NP-complete for Eulerian digraphs.
\end{corollary}

With a bit of work we can extend our result further to k-regular digraphs.

\begin{corollary}

Problem \ref{ioB} is NP-complete for 2-arc-strong k-regular digraphs.

\end{corollary}

\begin{proof}

In order to prove NP-completeness we will reduce from Problem \ref{ioB} for 2-arc-strong 2-regular digraphs.

Let the digraph $H$ have vertex set $\{b,c\}$ and arc set $\{bc,bc,cb,cb\}$.

Given a 2-arc-strong 2-regular digraph $D$ construct $D'$ by doing the following for every vertex $a\in V$. Split $a$ into $a^+$ and $a^-$, add two copies of the arc $a^-a^+$, add a copy of $H$ and finally add $k-2$ copies of the arcs $a^-b$, $ba^+$, $a^+c$ and $ca^-$, see figure \ref{fig:kMaker}. Notice that $D'$ is now a 2-arc-strong k-regular digraph.

\begin{figure}[htb]
\centering
\includegraphics[width=0.45\textwidth]{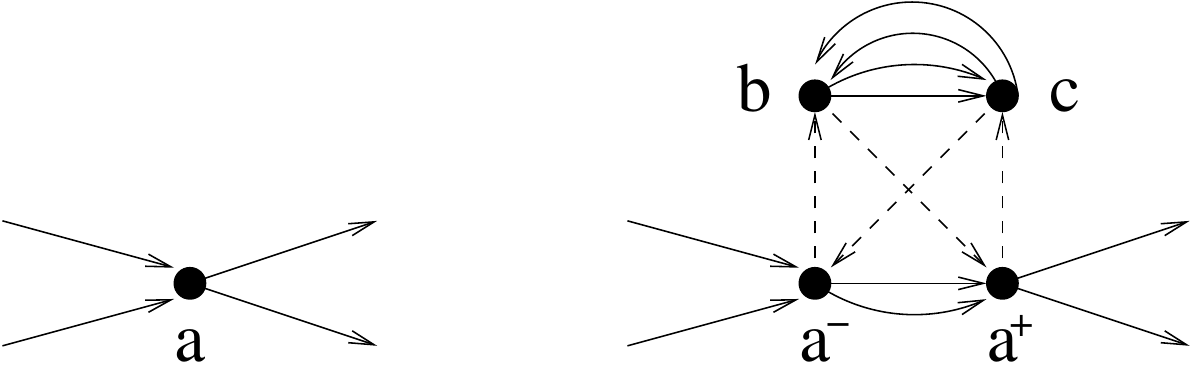}
\caption{Construction of a k-regular digraph from a 2-regular one. The dashed arcs represent $k-2$ parallel arcs.}
\label{fig:kMaker}
\end{figure}

An arc-disjoint pair of an out-branching $B_u^+$ and an in-branching $B_v^-$ in $D$ can obviously be extended to an in- and out-branching pair in $D'$.

Conversely given arc-disjoint branchings $B_u^+$ and $B_v^-$ in $D'$, simply undoing the steps we followed to construct $D'$ from $D$ will give a pair of branchings in $D$.
\end{proof}

\section{Arc-disjoint out-branchings and spanning-trees}\label{BTsec}

Since Problem \ref{ioB} remained NP-complete for 2-regular digraphs one might wonder whether this is still the case for the mixed version, where the in-branching is replaced by a spanning tree in $UG(D)$.

We will show that Problem \ref{oBsT} is  polynomially decidable for 2-regular digraphs, provided that the vertex $u$ is not specified in advance. In fact we will consider a more general problem in  k-regular digraphs.

\begin{problem}\label{kBT}

Given a $k$-regular digraph $D$ and ${0<l<k}$ decide whether $UG(D)$ contains a collection of  $k$-edge-disjoint spanning trees ${T_1,\ldots{},T_k}$, such that ${T_1,\ldots{},T_l}$ are out-branchings in $D$ (whereas we do not demand this from the remaining trees).

\end{problem}

For $l=0$ respectively $l=k$ we would get the purely undirected respectively purely directed problem characterized in Theorem \ref{tut} and Theorem \ref{root} respectively, where we referenced polynomial algorithms that solve both pure problems. The simplest mixed problem (with $k=2$ and $l=1$) on the other hand is already NP-complete for general digraphs \cite{bangTCSta}. But for k-regular digraphs we have:

\begin{theorem} \label{kequal}

Let $D=(V,A)$ be a k-regular digraph then $D$ contains k arc-disjoint out-branchings (with no restrictions on the roots) if and only if it contains k arc-disjoint spanning-trees.

\end{theorem}

\begin{proof}

Since every out-branching is also a spanning tree necessity is obvious.

For sufficiency assume that $D$ contains the k spanning trees $T_1,\ldots,T_k$. Since $|A(T_i)|=|V|-1$ we observe that the graph $D'=D\langle \bigcup_{i=1}^k{T_i}\rangle$, that is the union of the spanning trees, has exactly $k$  arcs fewer than $D$. We denote these arcs by ${u_1v_1,\ldots,u_kv_k}$ so that ${D'=D-\{u_1v_1,\ldots,u_kv_k\}}$.

Let $r(X)=\sum_{i=1}^k{|X\cap \{v_i\}|}$ for $X\subseteq V$ and notice that now every vertex $v$ in $D'$ has ${d^-_{D'}(v)=k-r(v)}$. We claim that $r$ is a root vector of $D'$, which would immediately imply the existence of k arc-disjoint out-branchings since ${r(V)=k}$. To prove $r$ is a root vector of $D'$ it is sufficient, by Theorem \ref{root}, to prove that ${d_{D'}^-(X)\geq k-r(X)}$ holds for all $X\subset V$.

Since $D'$ is the union of k spanning trees we have ${|A_{D'}(X)|\leq k|X|-k}$ for all ${X\subseteq V}$. On the other hand we know the in-degrees in $D'$ so we can give the exact number of arcs as ${|A_{D'}(X)|=k|X|-d^-_{D'}(X)-r(X)}$.

Combining this we get
\begin{align*} k|X|-d^-_{D'}(X)-r(X)=|A_{D'}(X)|&\leq k|X|-k\\ \Rightarrow k-r(X) &\leq d^-_{D'}(X)\\ \end{align*}
So $r$ is a root vector and since $r(V)=k$ Theorem \ref{root} gives that there exist k arc-disjoint out-branchings rooted in $D'$.
\end{proof}

This result implies that every k-regular digraph that contains a solution to Problem \ref{kBT} regardless of $l$ (where no roots are fixed) also contains k arc-disjoint out-branchings, since the solution to Problem \ref{kBT} contains k arc-disjoint spanning trees. So we can decide the problem by either employing Frank's \cite{Frank} algorithm for finding k arc-disjoint out-branchings without prescribed roots, or we employ Edmonds' algorithm \cite{Edmonds2} to find k arc-disjoint spanning trees, locate  the roots as in the proof of Theorem \ref{kBT} and then use Edmonds' branching algorithm \cite{Edmonds} to find k arc-disjoint out-branching with these roots.

\begin{corollary}
Problem \ref{kBT} is polynomially solvable.
\end{corollary}

One might be tempted to assume that regularity could be replaced by Eulericity in Theorem \ref{kequal}, but the digraph in Figure \ref{fig:Counter1} shows that the theorem does not hold for Eulerian digraphs.

\begin{figure}[htb]
\centering
\includegraphics[width=0.15\textwidth]{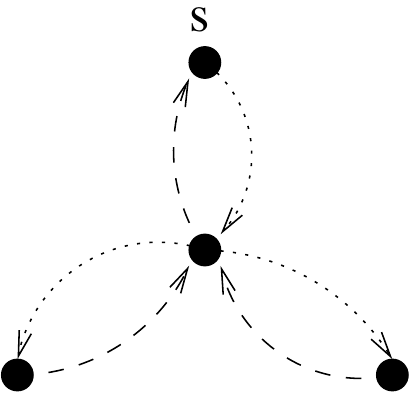}
\caption{The out-branching $B_s^+$ shown with dotted arcs and the spanning tree $T$ shown with dashed. Obviously it is not possible to find 2 arc-disjoint out-branchings in this digraph.}
\label{fig:Counter1}
\end{figure}

If we try to fix the roots,  the k arc-disjoint out-branchings provided by  Theorem \ref{kequal} no longer necessarily constitute a solution as witnessed by the example in Figure \ref{fig:Cond1}. We can find two arc-disjoint out-branchings rooted at $u$ and $v$ respectively, but the indicated partition makes it obvious that the digraph allows for no out-branching rooted at $s$ that leaves a connected subdigraph after removal.

\begin{figure}[htb]
\centering
\includegraphics[width=0.45\textwidth]{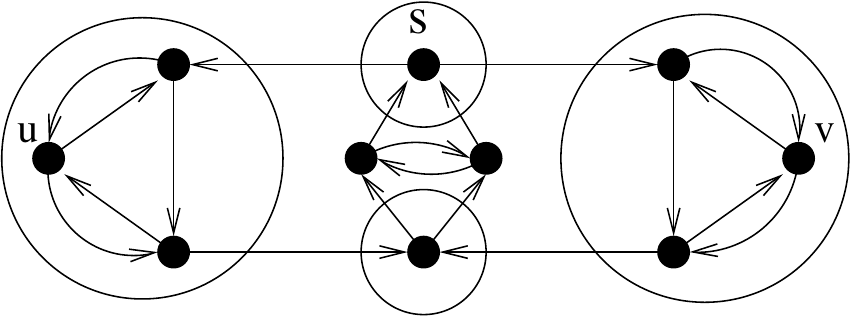}
\caption{Every out-branching rooted at $s$ will always disconnect this digraph upon removal.}
\label{fig:Cond1}
\end{figure}

\section{Related NP-complete problems for 2-regular digraphs} \label{related}

While we saw above that there is a polynomial algorithm for deciding, in a 2-regular digraph $D$, the existence  of an out-branching $B^+_s$ (whose root is not specified) such that $D-A(B^+_s)$ is connected, we could not resolve the question when we insist that $s$ is specified. In what follows we will show that problems which are closely related are indeed NP-complete for 2-regular digraphs.

Let $F$ be the digraph on 3 vertices $\{\alpha,\beta,\gamma\}$ and arcs $\{\alpha\beta, \beta\gamma, \gamma\alpha,\alpha\gamma,\gamma\beta\}$ (so $F$ is obtained from the complete digraph on 3 vertices $\{\alpha,\beta,\gamma\}$ by deleting the arc $\beta\alpha$).

\begin{theorem} \label{allnpc}

The following problems are NP-complete for 2-regular digraphs $D$

\begin{enumerate}

\item[(P1)] Given distinct vertices $s,t$; does $D$ have an $(s,t)$-path $P$ such that
$D-A(P)$ is connected?

\item[(P2)] Given distinct vertices $s,t$; does $D$ have an $(s,t)$-path $P$ such that
$D-A(P)$ is strongly connected?

\item[(P3)] Given distinct vertices $s,t$; does $D$ have an $(s,t)$-path $P$ such that
$D-A(P)$ contains an out-branching rooted at  $s$?

\end{enumerate}

\end{theorem}

\pf

We will show how to reduce 3-SAT to each of the problems (P1)-(P3)  in polynomial time. We will use almost the same reduction for all three problems. The reduction is similar to one used in \cite{bangTCSta}.

Let $W[u,v,p,q]$ be the digraph (the variable gadget) with vertices $\{u,v,y_1,y_2,\dots{}y_p,z_1,z_2,\ldots z_q\}$ and the arcs of the two (directed) $(u,v)$-paths $uy_1y_2\ldots{}y_pv$, $uz_1z_2\ldots{}z_qv$.

Let ${\cal F}$ be an instance of 3-SAT with variables $x_1,x_2,\ldots{},x_n$  and clauses $C_1,C_2,\ldots{},C_m$. We may assume that each variable $x$ occurs at least once either in the negated form or non-negated in $\cal F$. The ordering of the clauses $C_1,C_2,\ldots{},C_m$ induces an ordering of the occurrences of a variable $x$ and its negation $\bar{x}$ in these. With each variable $x_i$ we associate a copy of $W[u_i,v_i,p_i,q_i]$ where $x_i$ occurs $p_i$ times and $\bar{x}_i$ occurs $q_i$ times  in the clauses of $\cal F$. Identify end vertices of these graphs by setting $v_i=u_{i+1}$ for $i=1,2,\ldots{},{n-1}$. Let $s=u_1$ and $t=v_n$.

Let $D'$ be the digraph obtained in this way. 

For each $i=1,2,\ldots{},m$ we associate the clause $C_i$ with three of the vertices  $V_i=\{a_{i,1},a_{i,2},a_{i,3}\}$ from the digraph $D'$ above as follows:

Assume $C_i$ contains variables $x_j,x_k,x_l$ (negated or not). If $x_j$ is not negated in $C_i$ and this is the $r$'th copy of $x_j$ (in the order of the clauses that use $x_j$), then we identify $a_{i,1}$ with $y_{j,r}$ and if $C_i$ contains $\bar{x}_j$ and this is the $g$'th occurrence of $\bar{x}_j$, then we identify $a_{i,1}$ with $z_{j,g}$. We make similar identifications for $a_{i,2},a_{i,3}$. Thus $D'$ contains all the vertices $a_{j,i}$ for ${j\in [m]}$, ${i\in [3]}$.

The following Claim was proven in \cite{bangTCSta} but we include the easy proof for completeness.

\begin{claim} \label{varpath}

$D'$ contains an $(s,t)$-path $P$ which avoids at least one vertex from $\{a_{j,1},a_{j,2},a_{j,3}\}$ for each $j\in [m]$ if and only if $\cal F$ is satisfiable.

\end{claim} 

\noindent{}{\bf Proof of the Claim:}

Suppose $P$ is an $(s,t)$-path that avoids at least one vertex from $\{a_{j,1},a_{j,2},a_{j,3}\}$ for each $j\in [m]$. By construction, for each variable $x_i$, $P$ traverses either the subpath $u_iy_{i,1}y_{i,2}\ldots{}y_{i,p_i}v_i$ or the subpath $u_iz_{i,1}z_{i,2}\ldots{}z_{i,q_i}v_i$. Now define a truth assignment by setting $x_i$ false precisely when the first subpath is traversed for $x_i$. This is a satisfying truth assignment for $\cal F$ since for any clause $C_j$ at least one literal is avoided by $P$ and hence becomes true by the assignment (the literals traversed become false and those not traversed become true). Conversely, given a truth assignment for $\cal F$ we can form $P$ by routing it through all the false literals in the chain of variable gadgets. This proves the claim.
\hfill$\diamond$

Let $D$ be the 2-regular digraph we obtain from $D'$ as follows:

For each clause  $C_i$ (to which we associated the vertices $V_i=\{a_{i,1},a_{i,2},a_{i,3}\}$ from the digraph $D'$ above) we add 3 copies $F_{i,1},F_{i,2},F_{i,3}$ of $F$, with the vertices of the $h$th copy denoted by $\{\alpha_{i,h},\beta_{i,h},\gamma_{i,h}\}$ for $h=1,2,3$. and the following arcs $\{a_{i,1} \alpha_{i,1},\beta_{i,1} a_{i,2}, a_{i,2} \alpha_{i,2}, \beta_{i,2} a_{i,3},a_{i,3} \alpha_{i,3},\beta_{i,3}  a_{i,1}\}$. Finally take two further copies $F_1,F_2$ with vertices $\{\alpha_1,\beta_1,\gamma_1\}$ and $\{\alpha_2,\beta_2,\gamma_2\}$ and add the arcs $\{t\alpha_1,t\alpha_2,\beta_1 s,\beta_2 s\}$.

Now it is easy to see that $D$ has an $(s,t)$-path $P$ such that removing the arcs of $P$ leaves a connected digraph if and only if $D'$ contains an $(s,t)$-path which avoids at least one vertex from each of the sets $V_i$, $i=1,2,\ldots{},m$ ($P$ cannot enter any copy $F'$  of $F$ since that would disconnect these vertices from the rest after removing the arcs of $P$). Now the claim implies that (P1) is NP-complete.

To prove that (P2) is NP-complete we just need to show that if $D-A(P)$ is connected for some $(s,t)$-path $P$, then it is also strongly connected.

This follows from the construction since each copy of $F$ is a strongly connected subgraph (on 3 vertices) and since at least one vertex of $V_i$, $i=1,2,\ldots{},m$ is left untouched by $P$ it follows that $D-A(P)$ contains a cycle $C$ through $s,t$ formed by the arcs of $A(D')-A(P)$ and the path $t\alpha_1\gamma_1\beta_1s$ such that $C$ contains at least one vertex from $V_i$ for $i=1,2,\ldots{},m$. Clearly, as $P$ does not enter any copy of $F$, we can attach $F_2$ and all of the copies $F_{i,j}$, $i\in [m],j\in [3]$ to this structure and obtain a strong spanning subdigraph of $D$ which is arc-disjoint from $P$. 

Finally, to prove that (P3) is NP-complete, we just need to observe that if $P$ is any $(s,t)$-path in $UG(D)$ such that $D-A(P)$ is connected then $P$ must be a directed $(s,t)$-path in $D'$ and then, as we saw above, $D-A(P)$ is strongly connected and thus contains an out-branching from $s$. Clearly, if $D-A(P)$ contains an out-branching $B^+_s$, then the $(s,t)$-path contained in $B^+_s$ must meet each of the sets $V_i$ for $i=1,2,\ldots{},m$, implying that $P$ avoids at least one vertex from each $V_i$ for $i=1,2,\ldots{},m$ so, by the claim,  $D$ is a yes instance for (P3) if and only if $\cal F$ is satisfiable.
\qed

Notice that the only viable undirected $(s,t)$-path in the proof above is still a directed $(s,t)$-path in $D'$, so the problem is also NPC for undirected $(s,t)$-paths. Similarly if we delete $F_1$ from the graph and insert the arc $ts$ instead we get a proof that the problems are NPC for directed cycles containing $s$ instead of $(s,t)$-paths.

\section{Concluding remarks}

We did not completely settle Problem \ref{oBsT} for 2-regular digraphs since we had to leave open the case when the vertex $u$ is fixed in advance:

\begin{problem} \label{ufixed2reg}

What is the complexity of Problem \ref{oBsT} when $D$ is 2-regular and the vertex $u$ is part of the input?

\end{problem}

The proof in \cite{bangTCSta} that Problem \ref{oBsT} is NP-complete for general digraphs involves constructing a digraph which is not the union of two arc-disjoint spanning trees. It seems difficult to modify that proof so that the digraph used is the union of two arc-disjoint spanning trees.

\begin{problem} \label{2n-2branchtree}

What is the complexity of Problem \ref{oBsT} when the input digraph $D=(V,A)$ is the union of two arc-disjoint spanning trees?

\end{problem}

Note that, just as Corollary \ref{useall} followed directly from Theorem \ref{NPinout}, we again have that if Problem \ref{ufixed2reg} is NP-complete, then so is Problem \ref{2n-2branchtree} but it may be the case that Problem \ref{2n-2branchtree} is NP-complete while Problem \ref{ufixed2reg} is still polynomially solvable.

The following result from \cite{bangTCSta} indicates  that Problem  \ref{ufixed2reg} could be NP-complete. Note that here the digraph $H$ has either $n$ or $n+1$ arcs since we need $n-1$ for the connected remainder.

\begin{theorem}\cite{bangTCSta}

It is NP-complete to decide whether a strongly connected 2-regular digraph $D$ contains a spanning strong subdigraph $H$ so that $D-A(H)$ is connected.

\end{theorem}

\bibliographystyle{abbrv}
\bibliography{bibfile,refs}

\begin{thebibliography}{10}

\bibitem{bangJCT51}
J.~Bang-Jensen.
\newblock {Edge-disjoint in- and out-branchings in tournaments and related path
  problems}.
\newblock {\em J. Combin. Theory Ser. B}, 51(1):1--23, 1991.

\bibitem{Bang}
J.~Bang-Jensen and G.~Gutin.
\newblock {\em Digraphs: Theory, Algorithms and Applications, Second Edition}.
\newblock Springer Verlag, 2009.

\bibitem{bangTCS410}
J.~Bang-Jensen and M.~Kriesell.
\newblock {Disjoint directed and undirected paths and cycles in digraphs}.
\newblock {\em Theoretical Computer Science}, 410:5138--5144, 2010.

\bibitem{bangC31}
J.~Bang-Jensen and M.~Kriesell.
\newblock {On the problem of finding disjoint cycles and dicycles in a
  digraph}.
\newblock {\em Combinatorica}, 31:639--668, 2011.

\bibitem{bangTCSta}
J.~Bang-Jensen and A.~Yeo.
\newblock {Arc-disjoint spanning sub(di)graphs in digraphs}.
\newblock {\em Theoretical Computer Science}, to appear.

\bibitem{Berczi}
K.~B\'{e}rczi and A.~Frank.
\newblock {Packing Arborescences.}
\newblock {\em Technical Reports, Egerv\'{a}ry Research Group on Combinatorial
  Optimization.}, 2009.

\bibitem{Edmonds2}
J.~Edmonds.
\newblock {Matroid partition.}
\newblock In {\em Mathematics of the Decision Sciences: Part 1 (G.B. Dantzig
  and A.F. Veinott, eds.)}, pages 335--345. American Mathematical Society,
  1968.

\bibitem{Edmonds}
J.~Edmonds.
\newblock {Edge-disjoint branchings}.
\newblock {\em Combinatorial Algorithms}, pages 91--96, 1973.

\bibitem{Frank}
A.~Frank.
\newblock {On disjoint trees and arborescences.}
\newblock {Algebraic methods in graph theory, Vol. I, Cont. Szeged 1978,
  Colloq. Math. Soc. Janos Bolyai 25, 159-169 (1981).}, 1981.

\bibitem{plesnikIPL8}
J.~Plesn{\'i}k.
\newblock {The {N}{P}-completeness of the {H}amiltonian cycle problem in planar
  digraphs with degree bound two}.
\newblock {\em Inf. Process. Lett.}, 8(4):199--201, 1979.

\bibitem{Thomassen}
C.~Thomassen.
\newblock {Configurations in graphs of large minimum degree, connectivity, or
  chromatic number.}
\newblock {Combinatorial mathematics, Proc. 3rd Int. Conf., New York/ NY (USA)
  1985, Ann. N. Y. Acad. Sci. 555, 402-412 (1989).}, 1989.

\bibitem{Tutte}
W.~Tutte.
\newblock {On the problem of decomposing a graph into $n$ connected factors.}
\newblock {\em J. Lond. Math. Soc.}, 36:221--230, 1961.

\end{thebibliography}

\end{document}